\documentclass[conference,letterpaper,10pt]{ieeeconf}

\IEEEoverridecommandlockouts

\DeclareFontFamily{OT1}{pzc}{}
\DeclareFontShape{OT1}{pzc}{m}{it}{<-> s * [0.900] pzcmi7t}{}
\DeclareMathAlphabet{\mathpzc}{OT1}{pzc}{m}{it}
\usepackage{mypackages}

\renewcommand{\bar}{\overline}
\newcommand{\mbf}{\mathbf}
\renewcommand{\mod}{\hspace*{-0.01cm}\mathrm{mod}}

\newcommand{\af}[1]{{\small\textsf{{#1}}}}

\DeclarePairedDelimiter\ceil{\lceil}{\rceil}

\let\oldnl\nl
\newcommand{\nonl}{\renewcommand{\nl}{\let\nl\oldnl}}

\newcommand{\beq}{\begin{equation}}
\newcommand{\eeq}{\end{equation}}
\newcommand{\beqa}{\begin{eqnarray}}
\newcommand{\eeqa}{\end{eqnarray}}
\newcommand{\beqan}{\begin{eqnarray*}}
\newcommand{\eeqan}{\end{eqnarray*}}

\newcommand{\Gfrak}{\mathbf{G}}

\newcommand{\Acal}{{\cal A}}

\newcommand{\Ecal}{{\cal E}}

\newcommand{\Hcal}{{\cal H}}

\newcommand{\Ncal}{{\cal N}}
\newcommand{\Ocal}{{\cal O}}

\newcommand{\Ucal}{{\cal U}}
\newcommand{\Vcal}{{\cal V}}

\newcounter{l1}
\newcounter{l2}
\newcounter{l3}
\newcommand{\bdotlist}{\begin{list}{$\bullet$}{}}
\newcommand{\bboxlist}{\begin{list}{$\Box$}{}}
\newcommand{\bbboxlist}{\begin{list}{\raisebox{.005in}{{\tiny
$\blacksquare$ \ \ }}}{}}
\newcommand{\bdashlist}{\begin{list}{$-$}{} }
\newcommand{\blist}{\begin{list}{}{} }
\newcommand{\barablist}{\begin{list}{\arabic{l1}}{\usecounter{l1}}}
\newcommand{\balphlist}{\begin{list}{(\alph{l2})}{\usecounter{l2}}}
\newcommand{\bAlphlist}{\begin{list}{\Alph{l2}.}{\usecounter{l2}}}
\newcommand{\bdiamlist}{\begin{list}{$\diamond$}{}}
\newcommand{\bromalist}{\begin{list}{(\roman{l3})}{\usecounter{l3}}}

\newtheorem{theorem}{Theorem}
\newtheorem{lemma}{Lemma}

\newtheorem{remark}{Remark}
\newtheorem{definition}{Definition}

\begin{document}

\title{A Private and Finite-Time Algorithm for Solving a Distributed \\ System of Linear Equations}

\author{Shripad Gade \and Ji Liu \and Nitin H. Vaidya
\thanks{SG is with ECE Department at University of Illinois at Urbana-Champaign, JL is with Faculty of ECE Department at Stony Brook University, and NHV is with Faculty of CS Department at Georgetown University. e-mails: \texttt{gade3@illinois.edu, ji.liu@stonybrook.edu, nitin.vaidya@georgetown.edu}. SG was funded by Joan and Lalit Bahl Fellowship (UIUC) and Siebel Energy Institute Grant.}%
}

\maketitle

\begin{abstract}
This paper studies a system of linear equations, denoted as $Ax = b$, which is horizontally partitioned (rows in $A$ and $b$) and stored over a network of $m$ devices connected in a fixed directed graph. %
We design a fast distributed algorithm for solving such a partitioned system of linear equations, that additionally, protects the privacy of local data against an honest-but-curious adversary that corrupts at most $\tau$ nodes in the network. %
First, we present TITAN, privaTe fInite Time Average coNsensus algorithm, for solving a general average consensus problem over directed graphs, while protecting statistical privacy of private local data against an honest-but-curious adversary.  %
Second, we propose a distributed linear system solver that involves each agent/devices computing an update based on local private data, followed by private aggregation using {TITAN}. %
Finally, we show convergence of our solver to the least squares solution in finite rounds along with statistical privacy of local linear equations against an honest-but-curious adversary provided the graph has weak vertex-connectivity of at least $\tau+1$.
We perform numerical experiments to validate our claims and compare our solution to the state-of-the-art methods by comparing computation, communication and memory costs.
\end{abstract}

\section{Introduction} \label{Sec:Introduction}

Consider a system of linear equations,
\begin{align}
Ax = b, \label{Eq:ProblemI}
\end{align}
where, $x \in \mathbb{R}^n$ is the $n-$dimensional solution to be learned, and $A \in \mathbb{R}^{p \times n}$, $b \in \mathbb{R}^{p \times 1}$ encode $p$ linear equations in $n-$variables. The system of linear equations is horizontally partitioned and stored over a network of $m$ devices. Each device $i \in \{1, \ldots, m\}$ has access to $p_i$ linear equations denoted by,
\begin{align}
A_i x = b_i, \label{Eq:ProblemLocal}
\end{align} 
where, $A_i \in \mathbb{R}^{p_i \times n}$ and $b_i \in \mathbb{R}^{p_i \times 1}$. For instance, in Fig.~\ref{fig:1} we show a network of $m=5$ nodes and horizontal partitioning of $p=15$ linear equations in $n=5$ variables using colored blocks. In this paper, we consider an honest-but-curious adversary that corrupts at most $\tau$ devices/nodes in the network and exploits observed information to infer private data. We are interested in designing a fast, distributed algorithm that solves problem~\eqref{Eq:ProblemI}, while, protecting privacy of local information $(A_i,b_i)$ against such an honest-but-curious adversary. 

Solving a system of linear algebraic equations is a fundamental problem that is central to analysis of electrical networks, sensor networking, supply chain management and filtering \cite{xiao2005scheme,kar2012distributed,williams2017linear}. Several of these applications involve linear equations being stored at geographically separated devices/agents that are connected via a communication network. The geographic separation between agents along with communication constraints and unavailability of central servers necessitates design of distributed algorithms. Recently, several articles have proposed distributed algorithms for solving \eqref{Eq:ProblemI},  \cite{mou2015distributed,liu2017asynchronous,liu2017exponential,yang2020distributed,jakovetic2020distributed,wang2019solving,alaviani2018distributed} to name a few. %
In this work, we specifically focus on designing private methods that protect sensitive and private linear equations at each device/agent.
\begin{figure}[t]
    \centering
    \includegraphics[width=0.8\linewidth]{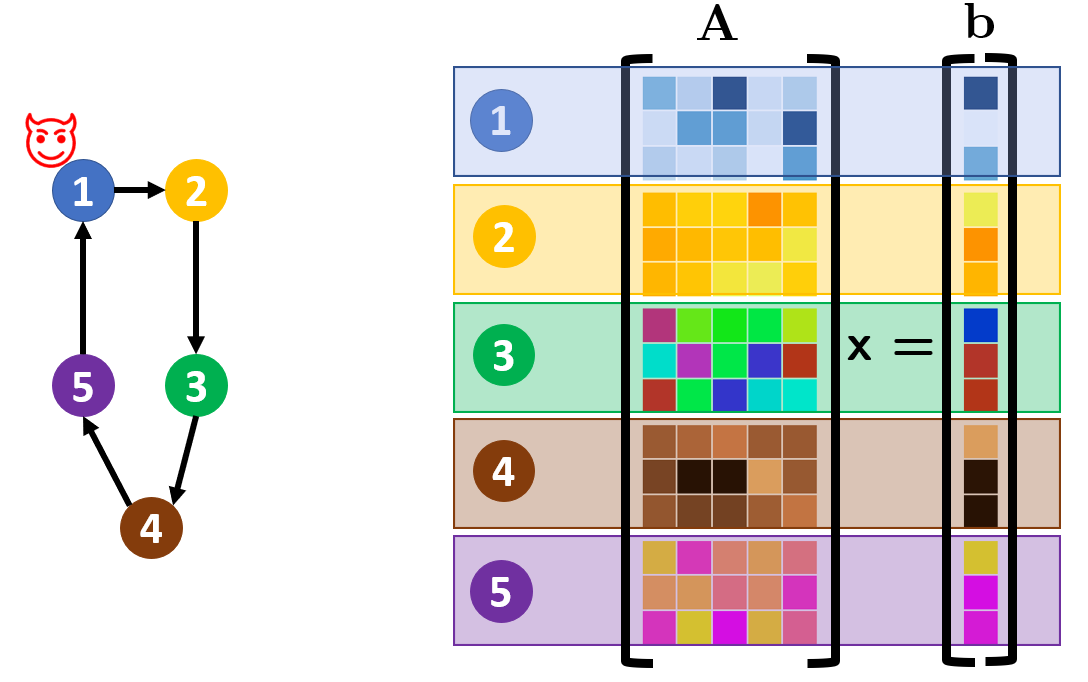}  
    \caption{Directed network with $m=5$ nodes and $1$ adversary. $A$ and $b$ are horizontally partitioned and stored at $m$ nodes. Local information is private to the nodes.}
    \label{fig:1}
\end{figure}

Literature has explored several approaches to solving a distributed system of linear equations. Authors in \cite{kar2012distributed} formulated the problem as a parameter estimation task. Consensus or gossip based distributed estimation algorithms are then used to solve \eqref{Eq:ProblemI}. Interleaved consensus and local projection based methods are explored in \cite{mou2015distributed,liu2017asynchronous,liu2017exponential}. These direct methods, involving feasilble iterates that move only along the null space of local coefficient matrix $A_i$, converge exponentially fast. One can also view solving \eqref{Eq:ProblemI} as a constrained consensus \cite{nedic2010constrained} problem, where agents attempt to agree to a variable $x$ such that local equations at each agent are satisfied. Problem~\eqref{Eq:ProblemI} can also be formulated as a convex optimization problem, specifically linear regression, and solved using plethora of distributed optimization methods as explored in \cite{yang2020distributed}. Authors augment their optimization based algorithms with a finite-time decentralized consensus scheme to achieve finite-time convergence of iterates to the solution. %
In comparison, our approach is not incremental and only needs two steps -- (a) computing local updates, followed by, (b) fast aggregation and exact solution computation. Our algorithm converges to the unique least squares solution in finite-time and additionally guarantees information-theoretic privacy of local data/equations $(A_i,b_i)$.

Few algorithms focus on privacy of local equations $(A_i,b_i)$. In this paper, we design algorithms with provable privacy properties. One can leverage vast private optimization literature by reformulating problem~\eqref{Eq:ProblemI} as a least-squares regression problem and use privacy preserving optimization algorithms \cite{gade2018acc,gade2018cdc,huang2015differentially,zhang2019admm,7431982,hale2017cloud} on the resulting strongly convex cost function. Differential privacy is employed in \cite{huang2015differentially,hale2017cloud} for distributed convex optimization, however, it suffers from a fundamental privacy - accuracy trade-off \cite{7431982}. Authors in \cite{zhang2019admm,alexandru2017privacy} use partially homomorphic encryption for privacy. However, these methods incur high computational costs and unsuitable for high dimensional problems. Secure Multi-party Computation (SMC) based method for privately solving system of linear equations is proposed in \cite{gascon2017privacy}, however, this solution requires a central Crypto System Provider for generating garbled circuits. In this work, we design a purely distributed solution. Liu {\em et al.} propose a private linear system solver in \cite{liu2018gossip1,liu2018gossip2}, however, as we discuss in Section~\ref{Sec:TITAN_LinSysDiscussion}, our algorithm is faster, requiring fewer iterations. Our prior work \cite{gade2018acc,gade2018cdc} proposes non-identifiability over equivalent problems as a privacy definition and algorithms to achieve it. It admits privacy and accuracy guarantees simultaneously, however, it uses a weaker adversary model that does not know distribution of noise/perturbations used by agents. 

In this paper, we consider a stronger definition of privacy viz. statistical privacy from \cite{gupta2019statistical, gupta2018information}. This definition of privacy allows for a stronger adversary that knows distribution of random numbers used by agents and has unbounded computational capabilities. We call this definition information-theoretic because additional observations do not lead to incremental improvement of adversarial knowledge about private data. Algorithms in \cite{gupta2019statistical, gupta2018information} provide algorithms for private average consensus over undirected graphs. We generalize their work to solve the problem over directed graphs and show a superior finite-time convergence guarantee.

\subsection*{Our Contributions:}

\noindent{\em Algorithm: }We present an algorithm, \af{TITAN} (priva\textbf{T}e f\textbf{I}nite \textbf{T}ime \textbf{A}verage co\textbf{N}sensus), that solves average consensus problem over directed graphs in finite-time, while protecting statistical privacy of private inputs. It involves a distributed {\em Obfuscation Step} to hide private inputs, followed by a distributed recovery algorithm to collect all perturbed inputs at each node. Agents then compute exact average/aggregate using the recovered perturbed inputs. We further leverage \af{TITAN} to solve Problem~\eqref{Eq:ProblemI} in finite time with strong statistical privacy guarantees. 

\vspace{0.5em}

\noindent{\em Convergence Results: }We show that \af{TITAN} converges to the exact average in finite time that depends only on the number of nodes and graph diameter. We show that graph being strongly connected is sufficient for convergence. Moreover, algorithm needs to know only an {\em an upper bound} on the number of nodes and graph diameter. We do not require out-degree of nodes to be known, a limitation commonly observed in push-sum type methods\cite{hadjicostis2015robust}, for solving average consensus over directed graphs. \af{TITAN} based solver converges to the unique least squares solution of \eqref{Eq:ProblemI} in finite time. 

\vspace{0.5em}

\noindent{\em Privacy Results: }We show that \af{TITAN} provides statistical privacy of local inputs as long as weak vertex connectivity of communication graph is at least $t+1$. This condition is also necessary and hence tight. Our privacy guarantee implies that for any two problems~\eqref{Eq:ProblemI} characterized by $(A,b)$ and $(A',b')$, such that rows of $(A,b)$ and $(A',b')$ stored at corrupted nodes are same and least squares solution for both systems is identical, the distribution of observations by the adversary is same. This equivalently entails adversary learning very little (statistically) in addition to the solution.


\section{Problem Formulation}
Consider a group of $m$ agents/nodes connected in a directed network. We model the directed communication network as a directed graph $\Gfrak=(\Vcal,\Ecal)$, where, $\Vcal = \{1,2,\ldots,m\}$ denotes the set of nodes and $\Ecal$ denote reliable, loss-less, synchronous and directed communication links. 

Recall, we are interested in solving a system of linear equations, problem~\eqref{Eq:ProblemI}, that is horizontally partitioned and stored at $m$ agents. Each agent $i$ has access to private $p_i$ linear equations in $n$ variables denoted by \eqref{Eq:ProblemLocal}. Equivalently our problem formulation states that each agent $i$ has access to private data matrices $(A_i,b_i)$ that characterize \eqref{Eq:ProblemLocal}. In this work, we assume that our system of linear equations in \eqref{Eq:ProblemI} admits a unique least squares solution, i.e. $A^T A$ matrix is full rank. If an exact solution for \eqref{Eq:ProblemI} exists, then it matches the least squares solution. Let $\af{ls-sol}(A,b)$ denote the unique least squares solution of $Ax = b$. We wish to compute $x^* \triangleq \af{ls-sol}(A,b)$ that solves the collective system $Ax = b$, while, protecting privacy of local data $(A_i,b_i)$. Next, we discuss the adversary model, privacy definition and few preliminaries.

\subsection{Adversary Model and Privacy Definition} \label{Sec:PrivDefinition}
We consider an honest-but-curious adversary, that follows the prescribed protocol, however, is interested in learning private information from other agents. The adversary can corrupt at most $\tau$ nodes and has access to all the information stored, processed and received by the corrupted nodes. Let us denote the corrupted nodes as $\Acal$. We assume the adversary and corrupted nodes have unbounded storage and computational capabilities. 

The coefficients of linear equations encode private and sensitive information. In the context of robotic or sensor networks, the coefficients of linear equations conceal sensor observations and measurements -- information private to agents. In the context of supply chain management and logistics, linear equations are used to optimize transport of raw-materials, and the coefficients of linear equations often leak business sensitive information about quantity and type of raw-materials/products being transported by a company. Mathematically, adversary seeks to learn local coefficient matrices $(A_i,b_i)$ corresponding to any non-corrupt agent $i$.  

Privacy requires that the observations made by the adversary do not leak significant information about the private inputs. We use the definition of information-theoretic or statistical privacy from \cite{gupta2018information,gupta2019statistical}. Let $\af{View}_\Acal((A,b))$ be the random variable denoting the observations made by set of adversarial nodes $\Acal$ given private inputs $(A,b)$.  We formally define statistical privacy (from \cite{gupta2018information,gupta2019statistical}) as follows.
\begin{definition}
A distributed protocol is $\Acal$-private if for all $(A,b)$ and $(A',b')$, such that $A[i,:] = A'[i,:]$, $b[i] = b'[i]$ for all $i \in \Acal$, and $\af{ls-sol}(A,b) = \af{ls-sol}(A',b')$, the distributions of $\af{View}_\Acal((A,b))$ and $\af{View}_\Acal((A',b'))$ are identical.
\end{definition}
\noindent Intuitively, for all systems of equations $(A',b')$, such that linear equations stored at $\Acal$ are the same and $\af{ls-sol}(A,b) = \af{ls-sol}(A',b')$, observations made by the adversary will have the same distribution, making them statistically indistinguishable from adversary $\Acal$. 

More, generally as discussed in \cite{gupta2018information,gupta2019statistical} for average consensus over private inputs $\{x_i\}_{i=1}^m$, an algorithm is $\Acal-$private, if for all inputs $\{x_i\}_{i=1}^m$ and $\{x'_i\}_{i=1}^m$, such that $x_i = x'_i$ for all $i\in\Acal$ and $\sum_{i=1}^m x_i = \sum_{i=1}^m x'_i$, the distributions of information observed by adversary is the same. In other words, probability density function of  $\af{View}_\Acal(\{x_i\}_{i=1}^m)$ and $\af{View}_\Acal(\{x'_i\}_{i=1}^m)$ are same.

\subsection{Notation and Preliminaries}

For each node $i \in \Vcal$, we define in-neighbor set, $\Ncal_i^{in}$, as the set of all nodes that send information to node $i$; and out-neighbor set, $\Ncal_i^{out}$, as the set of all nodes that receive information from node $i$. Let $\delta(\Gfrak)$ denote the diameter of graph $\Gfrak$, and $B$ denote the incidence matrix of graph $\Gfrak$. Let $\{\perp\}^m$ denote a $m-$dimensional empty vector. Let $\Ucal[a,b)$ be uniform distribution over $[a,b)$.

Modular arithmetic, typically defined over a finite field of integers, involves numbers wrapping around when reaching a certain value. In this work, we use real numbers and define an extension of modular arithmetic over reals:
\begin{definition} \label{Def:Modulo}
Consider a real interval $[0,a)\in \mathbb{R}$. We define $\mod(x,a)\in [0,a)$ as the remainder obtained when x is divided by a and the quotient is an integer. That is, $\mod(x,a) = x - pa$ where $p$ is the unique integer such that $\mod(x,a) \in [0,a)$.
\end{definition}
\noindent Modulo operator satisfies useful properties as detailed below. The proofs are easy and omitted for brevity.
\begin{remark}\label{Rem:PropModulo}
Modular arithmetic over reals satisfies following properties for all real numbers $\{y_i\}_i$ and integer $q$.
\begin{enumerate}
\item $\mod(\sum_{i=1}^q y_i,a) = \mod(\sum_{i=1}^q\mod(y_i,a),a)$,
\item $\mod(-y_i,a) = a - \mod(y_i,a).$
\end{enumerate}
\end{remark}


\section{\af{TITAN} - Private Average Consensus}
In this section, we develop \af{TITAN}, an algorithm for solving distributed average consensus with provable statistical privacy and finite-time convergence. In Section~\ref{Sec:TITAN_LinSys}, we will use \af{TITAN} to solve Problem~\eqref{Eq:ProblemI} with statistical privacy of local data $(A_i,b_i)$ and finite-time convergence to $x^*$.

Consider a simple average consensus problem over $m$ agents connected using directed graph $\Gfrak$. Each node $i \in \Vcal$ has access to private input $x_i \in [0,a)$. The objective is to compute average $(1/m) \sum_{i=1}^m x_i$, while, protecting statistical privacy of inputs $x_i$ (see Section~\ref{Sec:PrivDefinition}).

\af{TITAN} involves an obfuscation step to hide private information and generate obfuscated inputs. This is followed by several rounds of \af{Top-k} consensus primitive for distributed recovery of perturbed inputs. Consequently, we exploit {\em modulo aggregate invariance} property of the obfuscation step and locally process perturbed inputs to arrive at desired average. The obfuscation step guarantees statistical privacy of inputs, while, the distributed recover and local computation of average are key to finite-time convergence of the algorithm. We detail each of the steps below.

\begin{algorithm}[!b]
\SetAlgoLined
\LinesNumbered
\SetKw{KwRet}{Return}
\KwIn{$\{x_i\}_{i=1}^m$, where, $x_i \in [0,a)$ $\forall i\in\Vcal$, $T$, $k$}
\KwOut{$\bar{x} = (1/m) \sum_{i=1}^m x_i$}
\BlankLine
\textbf{Initialization} Node $i$ initializes $r_i = \{\perp\}^m$
\BlankLine
\nonl\textbf{Obfuscation Step} \\
Node $i$ sends random numbers $r_{ij} \sim \Ucal[0,ma)$ to each out-neighbor $j\in\Ncal_i^{out}$ \\
Node $i$ constructs perturbation $t_i$,
\begin{equation} 
t_i = \mod \left( \sum_{j \in \Ncal_i^{in}} r_{ji} - \sum_{j \in \Ncal_i^{out}} r_{ij}, ma \right) \label{Eq:Mask} 
\end{equation} \\
Each node $i$ perturbs private input,
\begin{align}
 \tilde{x}_i = \mod(x_i+t_i, ma).   \label{Eq:PerturbedStates} 
\end{align} \\
\nonl\textbf{Distributed Recovery using \af{Top-k} Primitive} \\
Each node $i\in\Vcal$ runs \af{Top-k} primitive $\ceil{m/k}$ times \\
\For{$t = 0, 1, \ldots, \ceil{m/k}-1$}{
$\left(r_i[tk+1 : (t+1)k], I_i[tk+1 : (t+1)k]\right) = \af{Top-k}\left((\{x_i\}_{i=1}^m\setminus r_i[1 : tk], \{i\}_{i=1}^m \setminus I_i[1:tk]), T \right)$ \\
}
\KwRet{$\bar{x} = (1/m) \mod(\sum_{l=1}^m r_i[l], ma)$.}
 \caption{\af{TITAN}$\left( \{x_i\}_{i=1}^m, T, k, m \right)$} \label{Algo:PrivateAvgConsensus}
\end{algorithm}

\subsection*{Obfuscation Step}
The {\em obfuscation step} is a distributed method to generate network correlated noise that vanishes under modulo operation over the aggregate. First, each node $i$ sends uniform random noise $r_{ij}\sim\Ucal[0,ma)$ to out-neighbors $\Ncal_i^{out}$ and receives $r_{li}$ from in-neighbors $l\in\Ncal_i^{in}$ (Line~2, Algorithm~\ref{Algo:PrivateAvgConsensus}). 

Next, each agent $i$ computes perturbation $t_i$ using \eqref{Eq:Mask} (Line~3, Algorithm~\ref{Algo:PrivateAvgConsensus}). Observe that due to the modulo operation, each perturbation $t_i$ satisfies $t_i \in [0,ma)$.

Finally, agent $i$ adds perturbation $t_i$ to its private value $x_i$ and performs a modulo operation about $ma$ to get the perturbed input, $\tilde{x}_i$, as seen in \eqref{Eq:PerturbedStates}. Observe that $\tilde{x}_i \in [0,ma)$. 

We now show the {\em modulo aggregate invariance} property of the obfuscation mechanism described above. Notice that each noise $r_{ij}$ is added by node $j$ to get $t_j$ (before modulo operation) and subtracted by node $i$ to get $t_i$ (before modulo operation). This gives us, 
\begin{align}
&\mod\left(\sum_{i=1}^m t_i , ma\right)  \nonumber\\
&\stackrel{(a)}{=} \mod\left(\sum_{i=1}^m \mod \left( \sum_{j \in \Ncal_i^{in}} r_{ji} - \sum_{j \in \Ncal_i^{out}} r_{ij}, ma\right), ma \right) \nonumber \\
&\stackrel{(b)}{=} \mod\left(\sum_{i=1}^m \left[\sum_{j \in \Ncal_i^{in}} r_{ji} - \sum_{j \in \Ncal_i^{out}} r_{ij} \right], ma \right) \nonumber \\
&\stackrel{(c)}{=} \mod(0,ma) = 0. \label{Eq:PerturbationsBalanced}
\end{align}
In the above expression, $(a)$ follows from definition of $t_i$ in \eqref{Eq:Mask}, $(b)$ follows from property 1 in Remark~\ref{Rem:PropModulo}, and $(c)$ is a consequence of perturbation design in \eqref{Eq:Mask}. We call this as {\em modulo aggregate invariance} of the obfuscation step. 

\subsection*{Distributed Recovery via \af{Top-k} Consensus Primitive}

We perform distributed recovery of perturbed inputs, that is, we run a distributed algorithm to ``gather'' all perturbed inputs ($\{\tilde{x}_i\}_{i=1}^m$) at each node. After the completion of this step, each node $i$ will have access to the entire set of perturbed inputs $\{\tilde{x}_i\}_{i=1}^m$. \af{Top-k} consensus primitive is a method to perform distributed recovery.

The \af{Top-k} consensus primitive is a distributed protocol for all nodes to agree on the largest $k$ inputs in the network. In \af{TITAN}, we run the \af{Top-k} consensus primitive and store the resulting list of top-$k$ perturbed inputs. We then run \af{Top-k} primitive again while excluding the perturbed inputs recovered from prior iterations. Executing \af{Top-k} consensus primitive successively $\ceil{m/k}$ times leads us to the list of all perturbed inputs in the network (Lines 5-8, Algorithm~\ref{Algo:PrivateAvgConsensus}). 

\vspace{1em}

\noindent {\em \af{Top-k} Consensus Primitive: }Recall, each node $i$ has access to a perturbed input $\tilde{x}_i$ and an unique identifier $i$. The \af{Top-k} consensus is a consensus protocol for nodes to agree over $k$ largest input values and associated node id's with ties going to nodes with larger id. Formally, the algorithm results in each node agreeing on $\{(x_{(1)}, \ldots, x_{(k)}), (i_{(1)}, \ldots, i_{(k)})\}$, where, $x_{(1)} \geq \ldots \geq x_{(k)} \geq x_{(k+1)} \geq \ldots \geq x_{(m)}$ is the ordering of private inputs $\{x_j\}_{j=1}^m$ and $i_{(k)}$ denote the ids corresponding to $x_{(k)}$ for each $k$. Ties go to nodes with larger id, implying, if $x_{(j)} = x_{(j+1)}$ then $i_{(j)} > i_{(j+1)}$.

Each node $i$ stores an estimate of $k$ largest inputs and their id's denoted by $L_i$ and $\ell_i$ respectively. These vectors, $L_i$ and $\ell_i$, are initialized with local private input and own agent id respectively (Line 1, Algorithm~\ref{Algo:TopkConsensus}). 

At each iteration $r = 1, \ldots, T$, agent $i$ share their estimates $L_i$ and $\ell_i$ to out-neighbors and receives $L_j$ and $\ell_j$ from in-neighbors $j \in \Ncal_i^{in}$ (Line 3, Algorithm~\ref{Algo:TopkConsensus}). Agent $i$ sets $L^+_i$ equal to the $k$ largest values of available inputs, that is $L_i$ and $\cup_{j \in \Ncal_i^{in}} L_j$; and sets $\ell^+_i$ as id's corresponding to the entries in $L^+_i$ (Lines 5,6, Algorithm~\ref{Algo:TopkConsensus}). In case of ties, that is several agents having the same input, the tie goes to agent with larger id. Next, agents update the local estimate of top $k$ entries as $L_i = L^+_i$ and $\ell_i = \ell^+_i$ (Line 8, Algorithm~\ref{Algo:TopkConsensus}). 

This process of selection of largest perturbed input and it's id is repeated $T$ times. As stated in Theorem~\ref{Th:TopKcorrectness}, each $L_i$ and $\ell_i$ converge to the {\em Top-k} values and associated id's respectively provided that $T \geq \delta(\Gfrak)$, the graph diameter.

\vspace{1em}

Recall, that running \af{Top-k} consensus primitive $\ceil{m/k}$ times, successively, allows each node to recover perturbed inputs $\{\tilde{x}_i\}_{i=1}^m$. Agents then add the perturbed inputs recovered by \af{Top-k} consensus algorithm and exploit  modulo aggregate invariance property of obfuscation step to exactly compute, $\bar{x} = (1/m) \mod(\sum_{l=1}^m r_i[l], ma)$.

\begin{algorithm}[t]
\LinesNumbered
\DontPrintSemicolon   
\SetKw{KwRet}{Return}
\KwIn{$\{(x_i,i)\}_{i=1}^m$, $T$, $k$}
\KwOut{$\{(x_{(1)}, \ldots, x_{(k)}), (i_{(1)}, \ldots, i_{(k)})\}$} \BlankLine
\textbf{Initialization} Node $i$ initializes two $k$ dimensional vectors $L_i = \{x_i, \perp, \ldots, \perp\}$ \& $\ell_i = \{i, \perp, \ldots, \perp\}$
\BlankLine
\For{$t = 1, \ldots, T$}{
Node $i$ sends $L_i$ \& $\ell_i$ to out-neighbors $\Ncal_i^{out}$ \;
\nonl Each node $i \in \Vcal$ performs: \;
\For{$r = 1, \ldots,k$}{
$L^+_i[r] = \max (L_i \cup_{j \in \Ncal_i^{in}} L_j \setminus (\cup_{s=1}^{r-1} L^+_i[s]))$  \;
$\ell^+_i[r] = \text{ id corresponding to } L_i[r]$ \;
$\qquad$ \text{\ldots \ Ties go to node with larger id} \;
$L_i = L^+_i$ and $\ell_i = \ell_i^+$
}
}
\BlankLine
\KwRet{$\{L_i, \ell_i\}$ $\rightarrow$ $\{(x_{(1)}, \ldots, x_{(k)}), (i_{(1)}, \ldots, i_{(k)})\}$}
\caption{\af{Top-k} consensus: \af{Top-k}($\{(x_i,i)\}_{i=1}^m$,$T$)} \label{Algo:TopkConsensus}
\end{algorithm}

\vspace{1em}

\subsection{Results and Discussion} \label{Sec:TITANDiscussion}

\noindent{\em Correctness Guarantee: }We first begin by a correctness result for the \af{Top-k} consensus primitive. It is a consequence of convergence of \af{max} consensus over directed graphs.
\begin{theorem} [Correctness of \af{Top-k}] \label{Th:TopKcorrectness}
If $\Gfrak$ is a strongly connected graph with diameter $\delta(\Gfrak)$ and $T\geq \delta(\Gfrak)$, then $L_i = L_j = \{x_{(1)}, \ldots, x_{(k)}\}$, $\ell_i = \ell_j = \{I_{(1)},\ldots,I_{(k)}\}$, for each $i,j\in \Vcal$, for the \af{Top-k} consensus algorithm.
\end{theorem}
\noindent The result establishes a lower bound on parameter $T$ for correctness of \af{Top-k} consensus primitive. If $\delta(\Gfrak)$ is not exactly known, we can set $T$ to be any upper bound on $\delta(\Gfrak)$, without worrying about correctness.  

The ability of \af{Top-k} protocol to recover $k$ largest perturbed inputs leads us to the correctness guarantee for \af{TITAN}. As a consequence of Theorem~\ref{Th:TopKcorrectness}, we can conclude that $\ceil{m/k}$ successive execution of \af{Top-k} primitive over perturbed inputs leads to recovery of all perturbed inputs at each node. Moreover, we use the aggregate invariant property of the obfuscation step, implying for all graphs $\Gfrak$, 
\begin{align}
    \mod\left(\sum_{i=1}^m \tilde{x}_i, ma\right) = \sum_{i=1}^m x_i.
\end{align}
We prove this in Section~\ref{Sec:Proofs}. Consequently, perturbed inputs allows us to compute the correct aggregate and average. The following result formally states this result: 
\begin{theorem}[Correctness of \af{TITAN}]
If $\Gfrak$ is strongly connected and $T \geq \delta(\Gfrak)$, then \af{TITAN} (Algorithm~\ref{Algo:PrivateAvgConsensus}) converges to the exact average of inputs $\bar{x} = \frac{1}{m} \sum_{i=1}^m x_i$ in finite time given by $T\ceil{m/k}$.  \label{Th:Correctness}
\end{theorem}

\noindent Note, for finite-time convergence, we only need $\Gfrak$ to be strongly connected. The time required for convergence is dependent only on the number of agents $m$ \footnote{If agents know only an upper bound on $\tilde{m} \geq m$, then we can run \af{TITAN} on inputs $x_i= 1$ for each $i\in\Vcal$ while using $\tilde{m}$ instead of $m$ and $a = 2$. Finally, by computing $\mod(\sum_i \tilde{x}_i, \tilde{m}a)$ we recover $m$ exactly.}, parameter $T$ (an upper bound on graph diameter $\delta(\Gfrak)$) and parameter $k$.

\vspace{0.5em}

\noindent{\em Privacy Guarantee: }The privacy guarantee is a consequence of the obfuscation step used in \af{TITAN}. Let $\bar{\Gfrak} = (\bar{\Vcal},\bar{\Ecal})$ be the undirected version of $\Gfrak$. More specifically, $\bar{\Gfrak}$ has the same vertex set $\bar{\Vcal} = \Vcal$ but the edge set $\bar{\Ecal}$ is obtained by taking all the edges in $\Ecal$ and augmenting it with the reversed edges. Consequently, $\bar{\Gfrak}$ is undirected. We define {\em weak vertex-connectivity} of a directed graph $\Gfrak$ as the vertex-connectivity of its undirected variant $\bar{\Gfrak}$. Weak vertex-connectivity of $\Gfrak$ is $\kappa(\bar{\Gfrak})$, where, $\kappa$ denotes the vertex-connectivity. We show that provided the weak vertex-connectivity of $\Gfrak$ is at least $\tau+1$, \af{TITAN} preserves statistical privacy of input. %
\begin{theorem}
If weak vertex-connectivity $\kappa(\bar{\Gfrak}) \geq \tau+1$, then \af{TITAN} is $\Acal$-private against any set of adversaries $\Acal$ such that $|\Acal| \leq \tau$.  
\label{Th:Privacy}
\end{theorem}
\noindent Note, we require $\Gfrak$ to be both strongly connected and possess weak vertex-connectivity of at least $\tau+1$ for achieving both finite-time correctness and statistical privacy guarantees. Moreover, the weak vertex-connectivity condition is also necessary and can be showed by contradiction, similar to Proof of Theorem~2 in \cite{gade2017private}. Consequently, the weak vertex-connectivity condition is {\em tight}.

\vspace{0.5em}
\noindent{\em Memory Costs: }\af{Top-k} primitive requires each node to maintain vectors $L_i$ and $\ell_i$ in addition to recovered pertrubed inputs. Overall the memory required per node is $(2k+m)d$ units, where, $d$ is the dimension of input $x_i$. This is larger as compared to standard average consensus methods and ratio consensus methods that require $d$ and $2d$ units respectively. Observe the trade-off between Memory Overhead and Convergence Time. As $k$ increases from $1$ to $m$, the convergence time decreases from $Tm$ to $T$, while the memory overhead (per node) increases from $(2+m)d$ to $3md$.  

\vspace{0.5em}
\noindent{\em Communication Costs: }The obfuscation step requires node $i$ to send $|\Ncal_i^{out}|d$ messages. Moreover, \af{Top-k} algorithm involves exchange of $L_i$ and $\ell_i$ by each node. This additionally requires $2k|\Ncal_i^{out}|T\ceil{m/k}d$ messages in total. Together, the communication overhead for node $i$ is $|\Ncal_i^{out}| \left( 2k T \ceil{m/k} + 1\right)d$. Total communication cost (per node) is largely independent of $k$, as total information exchanged over entire execution does not change with $k$. 

\vspace{0.5em}
\noindent{\em Comparison with \af{FAIM}: } Oliva {\em et al.} propose \af{FAIM}, finite-time average-consensus by iterated max-consensus, in \cite{oliva2016distributed}. \af{TITAN} has several similarities with \af{FAIM} and we can recover a statistically private version of \af{FAIM} by setting $k=1$. Our work, in addition to finite-time average consensus, is directed toward provably privacy of local information.


\section{Private Solver for System of Linear Equations} \label{Sec:TITAN_LinSys}

In this section, we develop a solver (Algorithm~\ref{Algo:PSLE}) employing \af{TITAN} to privately solve problem~\eqref{Eq:ProblemI}.

\begin{algorithm}[!b]
\SetAlgoLined
\LinesNumbered
\SetKw{KwRet}{Return}
\KwIn{$\{A_i^T A_i, A_i^T b_i\}_{i=1}^m$ \text{and} $a > \text{the largest entry of }[A_i^TA_i\; A_i^T b_i]$ over all $i\in\{1,2,\ldots,m\}$}
\KwOut{$x^*$}
\BlankLine
\nonl Run \textbf{\af{TITAN}} over each entry of $A_i^T A_i$ and $A_i^T b_i$\\
\BlankLine
Compute: $X$ = \af{TITAN}($\{A_i^T A_i\}_{i=1}^m, T, k, m$) \\
Compute: $Y$ = \af{TITAN}($\{A_i^T b_i\}_{i=1}^m, T, k, m$) \\
\KwRet{$x^* = (m X)^{-1} (m Y)$.}
\caption{Private Solver for Problem~\eqref{Eq:ProblemI}} \label{Algo:PSLE}
\end{algorithm}

The least squares solution to system of linear equations \eqref{Eq:ProblemI} can be expressed in closed form as, $(A^T A)^{-1} A^T b$. If an exact solution to \eqref{Eq:ProblemI} exists then the least squares solution matches it. Moreover, as the equations are horizontally partitioned, we can rewrite, 
\begin{align}
A^TA = \sum_{i=1}^m A_i^T A_i \quad A^Tb = \sum_{i = 1}^m A_i^T b_i.
\end{align}
Consequently, solution to system of linear equations can be computed by privately aggregating $A_i^T A_i$ and $A_i^T b_i$ separately over the network and computing,
\begin{align}
x^* = \left(\sum_{i=1}^m A_i^T A_i \right)^{-1} \left( \sum_{i=1}^m A_i^T b_i\right). \label{Eq:Solution}
\end{align}
Privately computing $x^*$ is equivalent to agents privately aggregating $A_i^T A_i$ and $A_i^T b_i$ over the directed network followed by locally computing $x^*$ using Eq.~\eqref{Eq:Solution}. 

\vspace{0.5em}

We assume, w.l.o.g, that each entry in matrix local updates $A_i^T A_i$ and $A_i^Tb_i$ lies in $[0,a)$, where, $a >$ largest entry in matrices $A_i^T A_i$ and $A_i^T b_i$. If this is not satisfied, we can add the same constant to each entry in the matrix and subtract it after computing aggregate. Next, in Algorithm~\ref{Algo:PSLE}, we run \af{TITAN} on each entry of matrices $\{A_i^T A_i\}_{i=1}^m$ to get update matrix $X$, run \af{TITAN} on each entry of vector $\{A_i^T b_i\}_{i=1}^m$ to get update vector $Y$. We select $T\geq \delta(\Gfrak)$ and a parameter $k \leq m$ following the discussion in Section~\ref{Sec:TITANDiscussion}. As a consequence of Theorem~\ref{Th:Correctness}, the algorithms terminate in finite time and $X = \frac{1}{m} \sum_{i=1}^m A_i^T A_i$ and $Y = \frac{1}{m}\sum_{i=1}^m A_i^T b_i$. Finally, we use \eqref{Eq:Solution} to compute the least squares solution. 

\subsection{Results and Discussion} \label{Sec:TITAN_LinSysDiscussion}
\noindent {\em Correctness and Privacy: }Our solution (Algorithm~\ref{Algo:PSLE}) involves using \af{TITAN} on each entry of matrices, $A_i^T A_i$ and $A_i^T b_i$. As a consequence of Theorem~\ref{Th:Correctness}, we know that provided $\Gfrak$ is strongly connected and parameter $T \geq \delta(\Gfrak)$, we get accurate estimate of $\sum_{i=1}^m A_i^T A_i$ and $\sum_{i=1}^m A_i^T b_i$ in finite time. Using \eqref{Eq:Solution}, we compute solution $x^*$, and as a result we have solved \eqref{Eq:ProblemI} accurately in finite time. From Theorem~\ref{Th:Privacy}, provided $\kappa(\bar{\Gfrak}) \geq \tau+1$, \af{TITAN} preserves the statistical privacy of local inputs $(A_i^T A_i,A_i^T b_i)$, equivalently preserves the statistical privacy of $(A_i,b_i)$, against any adversary that corrupts $\Acal$ subject to $|\Acal|\leq \tau$. 

\vspace{0.5em}
\noindent {\em Comparison with Relevant Literature: }%
Liu {\em et al.} propose a privacy mechanism, an alternative to the obfuscation mechanism in \af{TITAN}, and augment it to gossip algorithms for privately solving average consensus. This private average consensus is used along with direct method \cite{mou2015distributed} to arrive at a private linear system solver. However, it requires agents to reach complete consensus between successive direct projection based steps. This increases the number of iterations needed to solve the problem and significantly increases communication costs as the underlying method \cite{mou2015distributed} is only linearly convergent. Distributed Recovery phase in \af{TITAN} also requires complete consensus, but we do it in finite time using \af{Top-k} primitive, and it is only performed once.

Authors in \cite{yang2020distributed} propose a finite-time solver for solving \eqref{Eq:ProblemI}. However, under this protocol an arbitrarily chosen node/agent observes states for all nodes, and the observations are used compute the exact solution. This algorithm was not designed to protect privacy of local information and consequently leads to large privacy violations by the arbitrarily chosen node. Algorithm~\ref{Algo:PSLE} solves \eqref{Eq:ProblemI} in finite time, while additionally protecting statistical privacy of local equations. 
Moreover, the algorithm in \cite{yang2020distributed} is computationally expensive -- requiring matrix singularity checks and kernel space computation. In comparison, our algorithm is inexpensive and the most expensive step is matrix inversion in \eqref{Eq:Solution}, that needs to be performed only once. 

Direct methods \cite{mou2015distributed,liu2017exponential}, constrained consensus applied to linear systems \cite{nedic2010constrained} and distributed optimization methods applied to linear regression \cite{shi2015extra,nedic2017achieving,sun2019convergence} are only linearly convergent, while, we provide superior finite-time convergence guarantee. 

Finite-time convergence of underlying consensus is critical for statistically private mechanisms such as the one in \af{TITAN} and the algorithms from \cite{gupta2018information,gupta2019statistical}. These mechanisms rely on modulo arithmetic, if we add them to a linearly convergent solver which outputs in-exact average/aggregate, then performing modulo operation over the in-exact output, in final step,  may arbitrarily amplify errors.

\vspace{0.5em}
\noindent{\em Improving Computational Efficiency: }In our approach, each node needs to compute an inverse, $\left( \sum_{i=1}^m A_i^T A_i \right)^{-1}$, which takes $\Ocal(n^3)$ computations. Note, this inverse is not performed on private data. We can reduce computational cost by allowing one node to perform the inversion followed by transmitting the solution to all nodes either by flooding protocol or sending it over a spanning-tree of $\Gfrak$.


\section{Numerical Experiments}
In this section, we perform two numerical experiments to validate Algorithm~\ref{Algo:PSLE}. First, we conduct a simple simulation over the problem defined in Fig.~\ref{fig:1} and show that update matrices $A_i^T A_i$ observed by the adversary appear to be random (Lemma~\ref{Lem:Randomness}). Second, we run a large scale experiment with synthetic data, with $m = 100$, $p=10000$ and $n=100$.

We have $m=5$ nodes, with $p=15$ linear equations in $n=5$ variables being stored on the 5 nodes (3 each) as shown in Fig.~\ref{fig:1}. We generated $A$ and $b$ matrices by drawing their entries from a Gaussian distribution (mean = 0 and variance = 2) and verified $A^TA$ is full rank. We executed Algorithm~\ref{Algo:PSLE} at each node. We select parameter $T = m \geq \delta(\Gfrak) = 4$, $k=m=5$, and $a = 20$. The algorithm solves the problem exactly in $T = 6$ iterations. Note $\Gfrak$ is strongly connected and has a weak vertex-connectivity of 2. Consequently, both accuracy and statistical privacy are guaranteed by Theorems~\ref{Th:Correctness}, \ref{Th:Privacy}. The perturbed update matrices generated after obfuscation step in \af{TITAN} are received by adversary node $1$ and shown in Fig.~\ref{fig:2}. The color of each entry in the matrix represents its numerical value and Fig.~\ref{fig:2} shows that the perturbed updates are starkly different from private updates and appear random. 

\begin{figure}[!t]
    \centering
    \includegraphics[width=0.85\linewidth]{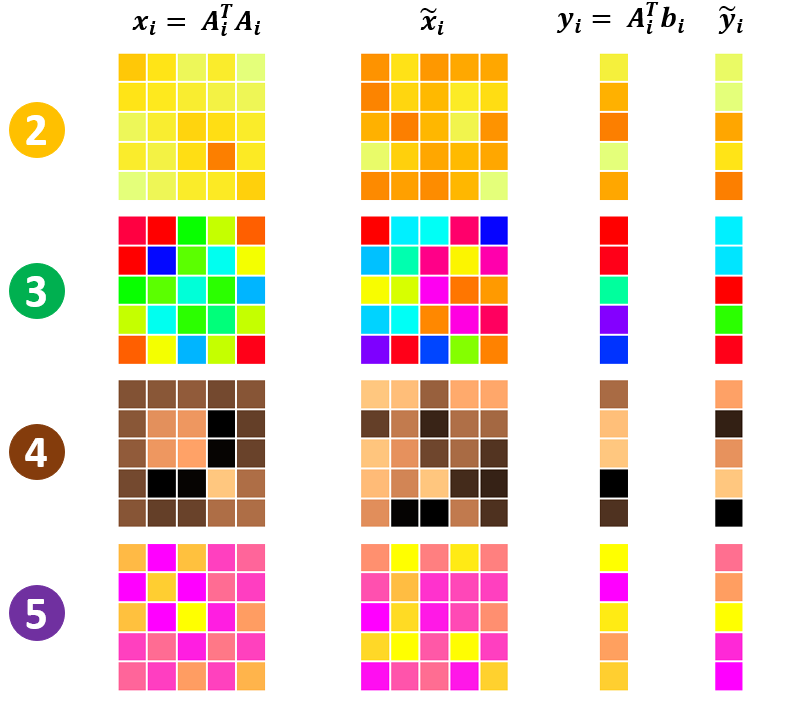}  
    \caption{Inputs $x_i = A_i^TA_i$ and $y_i = A_i^T b_i$ for $i\in\Vcal$ and perturbed inputs $\Tilde{x}_i$ and $\Tilde{y}_i$. The perturbed inputs appear to be random and very different from the private inputs.}
    \label{fig:2}
\end{figure}

Consider a large scale system with $m=100$ agents and $p=10000$ equations in $n=100$ variables that are horizontally partitioned for each agent to have $p_i = 100$ equations ($\forall i \in\Vcal$). The coefficients of the linear equations are synthetically generated via a Gaussian process and admit a unique least squares solution. The graph $\Gfrak$ is a directed ring with graph diameter $\delta(\Gfrak) = m-1=99$ and $\kappa(\Bar{\Gfrak}) = 2$. We run Algorithm~\ref{Algo:PSLE} with parameters $T = m \geq \delta(\Gfrak)$ and $k = 10$. We consider an honest-but-curious adversary that corrupts at most one agent and from $\kappa(\Bar{\Gfrak}) =2$ we guarantee statistical privacy of local data $(A_i,b_i)$. The algorithm converges to the solution in $1000$ iterations.


\section{Analysis and Proofs} \label{Sec:Proofs}

\subsection{Convergence Analysis}
\noindent We first prove correctness of \af{Top-k} consensus primitive.
\begin{proof}[Proof of Theorem~\ref{Th:TopKcorrectness}] In the \af{Top-k} consensus algorithm, each agent/node tries to keep track of the largest $k$ inputs observed till then. As $T\geq \delta{(\Gfrak)}$, each one of the $k$ largest inputs, i.e., $x_{(1)}, \ldots, x_{(k)}$, reaches each node in the network. Consequently, each nodes' local states $L_i$ and $\ell_i$ converge to the $k$ largest inputs in the network and associated id's. 
\end{proof}
\noindent Next, we prove correctness of \af{TITAN} in finite-time.
\begin{proof} [Proof of Theorem~\ref{Th:Correctness}] The correctness result in Theorem~\ref{Th:Correctness} follows from two key statements: (1) \af{TITAN} output is exactly equal to $\bar{x} = (1/m)\sum_{i=1}^m x_i$, and (2) \af{TITAN} converges in finite time. We prove both the statements above.

\vspace{0.5em}

\noindent \textbf{(1)} Recall that \af{TITAN} outputs $(1/m)\mod(\sum_i \Tilde{x}_{i=1}^m, ma)$. We use properties of modulo function to get,
\begin{align}
    &\mod(\sum_i \tilde{x}_{i=1}^m, ma) = \mod(\sum_{i=1}^m \mod(x_i+t_i, ma), ma) \nonumber \\
    &\stackrel{(a)}{=} \mod(\sum_{i=1}^m (x_i+t_i), ma) \stackrel{(b)}{=}\mod(\sum_{i=1}^m x_i, ma) = \sum_{i=1}^m x_i. \label{Eq:T2Proof1}
\end{align}
Recall, (a) follows from Remark~\ref{Rem:PropModulo}, (b) follows from $\sum_{i=1}^m t_i = 0$, and final equality follows from $x_i \in [0,a)$, $\sum_{i=1}^m x_i \in [0,ma)$ and Definition~\ref{Def:Modulo}. We have proved \af{TITAN} computes the exact aggregate and consequently the average.

\vspace{0.5em}

\noindent \textbf{(2)}
The \af{Top-k} protocol involves $T=\tilde{\delta} \geq \delta(\Gfrak)$ iterations where nodes share the largest $k$ values that they have encountered. From Theorem~\ref{Th:TopKcorrectness}, \af{Top-k} converges to the largest $k$ perturbed inputs. We need to run $\lceil m/k\rceil$ iterations of \af{Top-k} consensus. Consequently, the total iterations for complete execution is $T \ceil{m/k}$.
\end{proof}

\subsection{Privacy Analysis}
The privacy analysis presented here is similar in structure to \cite{gupta2018information,gupta2019statistical}. The key difference lies in the graph condition required for privacy. Recall $\bar{\Gfrak}$ is augmented $\Gfrak$. Specifically, $\bar{\Gfrak}$ has the same vertex set $\Vcal$ but the edge set $\bar{\Ecal}$ is obtained by taking all edges in $\Ecal$ and augmenting it with reversed edges. 
Recall, the noise shared on edge $e = (i,j) \in \Ecal$, is denoted as $r_{ij}$ and it is uniformly distributed over $[0,ma)$. Perturbations $t_i$ constructed using \eqref{Eq:Mask} can be written as $\mbf{t} =\mod(B \mbf{r}, ma)$, where $B$ is the incidence matrix of $\Gfrak$ and $\mbf{r}$ is the vector of $r_{ij}$ ordered according to the edge ordering in columns of $B$. 
If $\Gfrak$ is strongly connected then $\Gfrak$ is weakly connected and $\bar{\Gfrak}$ is connected. 
\begin{lemma}
If $\bar{\Gfrak}$ is connected then $\mbf{t}$ is uniformly distributed over all points in $[0,ma)^m$ subject to the constraint $\mod(\sum_{i=1}^m t_i, ma) = 0$.\label{Lem:Mask}
\end{lemma}
\begin{proof} 
Recall that perturbation vector $\mbf{t} = [t_i]$ can be written as $\mbf{t} = \mod(\sum_{j=1}^{|\Ecal|} B r_{j}, ma)$, where the modulo operation is performed element-wise. 

The connectivity of $\bar{\Gfrak}$ ensures that each $t_i$ is a linear combination of uniform random perturbations ($r_j$'s). We use the fact that $\mod(a+b)$ is uniformly distributed if either $a$ or $b$ is uniformly distributed. Using above statements along with $r_j \sim \Ucal[0,ma)$ we get $t_i \sim \Ucal[0,ma)$. $\mathbf{t}$ is uniformly distributed over $[0,ma)^m$. Moreover, from \eqref{Eq:PerturbationsBalanced}, we know, $\mod(1^T \mathbf{t}, ma) =  0$, is guaranteed if $\mbf{t} = \mod(B \mbf{r}, ma)$. This completes the proof of Lemma~\ref{Lem:Mask}.
\end{proof}

Using the above property of perturbations, $\mbf{t}$, we can show that perturbed inputs appear to be uniformly random, as described in the next lemma.
\begin{lemma}
If $\bar{\Gfrak}$ is connected then the effective inputs $\tilde{x}_i$ are uniformly distributed over $[0,ma)$ subject to the constraint $\mod(\sum_{i=1}^m \tilde{x}_i, ma) = \sum_i x_i$. \label{Lem:Randomness}
\end{lemma}
\begin{proof} 
Let $\Tilde{X}, X$ and $T$ represent the random vectors of agents obfuscated inputs, private inputs and perturbations respectively. Let $f_{\Tilde{X}}, f_{X}$ and $f_T$ denote the probability distribution of the respective random variables.

\noindent Recall $\Tilde{x}_i = \mod(x_i + t_i)$ and the fact that $t_i$ and $x_i$ are independent. We have,
\begin{align*}
    f_{\Tilde{X}}(\Tilde{x}|X=x) = f_{T}(T=\mod(\Tilde{x}-x))
\end{align*}
As $\bar{\Gfrak}$ is connected and $\mod(\sum_{i=1}^m \Tilde{x}_i) = \sum_{i=1}^m x_i$ from \eqref{Eq:T2Proof1}, we know that $\mod(\Tilde{x} - x)$ is uniformly distributed over $[0,ma)^m$. And we have $f_{T}(T=\mod(\Tilde{x}-x)) = $ constant for any $\Tilde{x} \in [0,ma)^m$, given $\mod(1^T\Tilde{x}) = 1^T x$.

\noindent We have that $f_{\Tilde{X}}(\Tilde{x}| X = x) = $ constant for all $\Tilde{x} \in [0,ma)^m$ implying that the perturbed input appears to be uniformly distributed over $[0,ma)^m$ subject to $\mod(1^T\Tilde{x}) = 1^T x$. 
\end{proof}

Recall, $\Acal$ is the set of honest-but-curious adversaries with $|\Acal| \leq \tau$. Also recall, adversarial nodes observe/store all information directly received and transmitted.

\begin{proof}[Proof of Theorem~\ref{Th:Privacy}]
Let $\Hcal = \Vcal \setminus \Acal$ denote the set of honest nodes. Let $\Gfrak_\Hcal = (\Hcal, \Ecal_\Hcal)$ denote the subgraph induced by honest nodes, implying, $\Ecal_\Hcal \subseteq \Ecal$ is the set of all edges from $\Gfrak$ that are incident on \& from two honest nodes. Let $B_\Hcal$ denote the oriented incidence matrix of graph $\Gfrak_\Hcal$. 

We know that vertex connectivity $\kappa(\bar{\Gfrak}) \geq \tau+1$, implying that deleting any $\tau$ nodes from $\bar{\Gfrak}$ does not disconnect it. Implying, $\bar{\Gfrak}_\Hcal$ is connected. 

The information accessible to an adversary, defined as \af{View}$_\Acal$, consists of the private inputs of corrupted agents, perturbed inputs of honest agents and the random numbers transmitted or received by corrupted nodes.
\begin{align*}
    \af{View}_{\Acal}(x) = \{\{\tilde{x}_i | i \in \Hcal\}, \{x_i | i \in \Acal\}, \{r_{ij} | i \in \Acal \text{ or } j \in \Acal\}\}
\end{align*}
For privacy, we prove that the probability distributions of $\af{View}_\Acal(x) = \af{View}_\Acal(x')$ for any two inputs $x, x'$ such that $x_i = x'_i$ for all $i \in \Acal$ and $\sum_{i \in \Vcal}x_i = \sum_{i \in \Vcal} x'_i$.

Let the incidence matrix be partitioned as $B = [B_\Hcal \; B_\Acal]$, where $B_\Hcal$ are columns of $B$ corresponding to edges in $\Ecal_\Hcal$ and $B_\Acal$ are columns of $B$ corresponding to edges in $\Ecal_\Acal = \Ecal \setminus \Ecal_\Hcal$. Let $B_e$ represent the $e^{th}$ column of $B$ and $B_{i,e}$ represent the $[i,e]^{th}$ entry of matrix $B$. $\mbf{r}_e$ denote the $e^{th}$ entry of vector $\mbf{r}$. Note that the perturbations can be expressed as, $t_i = \mod \left(\sum_{e \in \Ecal_\Hcal \cup \Ecal_\Acal} B_{i,e} \mbf{r}_e, ma \right)$ and equivalently as $t_i = \mod(t_i^{\Hcal} + t_i^{\Acal},ma)$, where, $t_i^\Hcal = \mod \left(\sum_{e \in \Ecal_\Hcal} B_{i,e} \mbf{r}_{e},ma \right)$ and $t_i^\Acal = \mod \left(\sum_{e \in \Ecal_\Acal}B_{i,e} \mbf{r}_{e},ma\right)$, following Remark~\ref{Rem:PropModulo}.

Using Lemma~\ref{Lem:Mask} we can state the following. The values $\{t^\Hcal_i\}_{i\in\Hcal} = \{\mod(\sum_{e \in \Ecal_\Hcal}B_{i,e} \mbf{r}_{e},ma)\}_{i \in \Hcal}$ lies in the span of columns of $B_{\Hcal}$ and is uniformly distributed over $[0,ma)^{|\Hcal|}$ subject to $\sum_{i \in \Hcal} \mod(\sum_{e \in \Ecal_\Hcal}B_{i,e} \mbf{r}_{e},ma) = 0$ given that  $\bar{\Gfrak}_\Hcal$ is connected. 
Consequently, the masks $\{t_i\}_{i\in\Hcal}$ are uniformly distributed over $[0,ma)^{|\Hcal|}$ subject to $\mod(\sum_{i \in \Hcal} t_i,ma) = -\mod(\sum_{i \in \Acal} t_i, ma)$, given $\bar{\Gfrak}_\Hcal$ is connected, and given perturbations $\{\mbf{r}_e\}_{e\in\Ecal_\Acal}$. 

Recall, $\mbf{r}_{e \in \Ecal_\Hcal}$ are uniformly and independently distributed in $[0,ma)$ given values of $\{\mbf{r}_{e}\}_{e \in \Ecal_\Acal}$ and $t_i$ for each $i \in \Acal$. 

We use the fact that $\bar{\Gfrak}_\Hcal$ is connected and  Lemma~\ref{Lem:Randomness}, we get that, $\tilde{x}_i$ are uniformly distributed over $[0,ma)$ subject to $\mod(\sum_{i}\tilde{x}_i, ma) = \sum_i x_i$. Thus, if $X_\Hcal$, $\tilde{X}_\Hcal$ are random variables representing $\{x_i\}_{i \in \Hcal}$ and $\{\tilde{x}_i\}_{i\in\Hcal}$, then,
\begin{align*}
    f_{\tilde{X}_\Hcal}(\tilde{x}_\Hcal|x_\Hcal, \{\mbf{r}_e\}_{e \in \Ecal_\Acal}) = \text{constant}
\end{align*}
for all $\tilde{x}_\Hcal$ in $[0,ma)^{|\Hcal|}$ that satisfy, $\mod(\sum_{i\in \Hcal}\tilde{x}_i, ma) = \mod (\sum_{i\in\Vcal}x_i - \sum_{i \in\Acal}\tilde{x}_i, ma)$.

We combine this with the fact that the perturbations $\mbf{r}_e$ are independent of inputs $x$, to get,
\begin{align*}
    f_{\af{View}_\Acal(x)}(\{\tilde{x}_\Hcal, \{\mbf{r}_e\}_{e \in\Ecal_\Acal}\}) = f_{\af{View}_\Acal(x')} (\{\tilde{x}_\Hcal, \{\mbf{r}_e\}_{e \in\Ecal_\Acal}\})
\end{align*}
$\forall x, x'$ such that, $x_i = x'_i$ $\forall i \in \Acal$ and $\sum_i x_i = \sum_i x'_i$.
\end{proof}

\section{Conclusion}
We presented \af{TITAN}, a finite-time, private algorithm for solving distributed average consensus. We show that \af{TITAN} converges to the average in finite-time that is dependent only on graph diameter and number of agents/nodes. It also protects statistical privacy of inputs against an honest-but-curious adversary that corrupts at most $\tau$ nodes in the network, provided weak vertex-connectivity of graph is at least $\tau+1$. We use \af{TITAN} to solve a horizontally partitioned system of linear equations in finite-time, while, protecting statistical privacy of local equations against an honest-but-curious adversary.

\bibliographystyle{IEEEtran}
\bibliography{Central}

\begin{thebibliography}{10}
\providecommand{\url}[1]{#1}
\csname url@samestyle\endcsname
\providecommand{\newblock}{\relax}
\providecommand{\bibinfo}[2]{#2}
\providecommand{\BIBentrySTDinterwordspacing}{\spaceskip=0pt\relax}
\providecommand{\BIBentryALTinterwordstretchfactor}{4}
\providecommand{\BIBentryALTinterwordspacing}{\spaceskip=\fontdimen2\font plus
\BIBentryALTinterwordstretchfactor\fontdimen3\font minus
  \fontdimen4\font\relax}
\providecommand{\BIBforeignlanguage}[2]{{%
\expandafter\ifx\csname l@#1\endcsname\relax
\typeout{** WARNING: IEEEtran.bst: No hyphenation pattern has been}%
\typeout{** loaded for the language `#1'. Using the pattern for}%
\typeout{** the default language instead.}%
\else
\language=\csname l@#1\endcsname
\fi
#2}}
\providecommand{\BIBdecl}{\relax}
\BIBdecl

\bibitem{xiao2005scheme}
L.~Xiao, S.~Boyd, and S.~Lall, ``A scheme for robust distributed sensor fusion
  based on average consensus,'' in \emph{IPSN 2005. Fourth International
  Symposium on Information Processing in Sensor Networks, 2005.}\hskip 1em plus
  0.5em minus 0.4em\relax IEEE, 2005, pp. 63--70.

\bibitem{kar2012distributed}
S.~Kar, J.~M. Moura, and K.~Ramanan, ``Distributed parameter estimation in
  sensor networks: Nonlinear observation models and imperfect communication,''
  \emph{IEEE Transactions on Information Theory}, vol.~58, no.~6, pp.
  3575--3605, 2012.

\bibitem{williams2017linear}
G.~Williams, \emph{Linear algebra with applications}.\hskip 1em plus 0.5em
  minus 0.4em\relax Jones \& Bartlett Learning, 2017.

\bibitem{mou2015distributed}
S.~Mou, J.~Liu, and A.~S. Morse, ``A distributed algorithm for solving a linear
  algebraic equation,'' \emph{IEEE Transactions on Automatic Control}, vol.~60,
  no.~11, pp. 2863--2878, 2015.

\bibitem{liu2017asynchronous}
J.~Liu, S.~Mou, and A.~S. Morse, ``Asynchronous distributed algorithms for
  solving linear algebraic equations,'' \emph{IEEE Transactions on Automatic
  Control}, vol.~63, no.~2, pp. 372--385, 2017.

\bibitem{liu2017exponential}
J.~Liu, A.~S. Morse, A.~Nedi{\'c}, and T.~Ba{\c{s}}ar, ``Exponential
  convergence of a distributed algorithm for solving linear algebraic
  equations,'' \emph{Automatica}, vol.~83, pp. 37--46, 2017.

\bibitem{yang2020distributed}
T.~Yang, J.~George, J.~Qin, X.~Yi, and J.~Wu, ``Distributed least squares
  solver for network linear equations,'' \emph{Automatica}, vol. 113, pp.
  1087--98, 2020.

\bibitem{jakovetic2020distributed}
D.~Jakovetic, N.~Krejic, N.~K. Jerinkic, G.~Malaspina, and A.~Micheletti,
  ``Distributed fixed point method for solving systems of linear algebraic
  equations,'' \emph{arXiv preprint arXiv:2001.03968}, 2020.

\bibitem{wang2019solving}
P.~Wang, S.~Mou, J.~Lian, and W.~Ren, ``Solving a system of linear equations:
  From centralized to distributed algorithms,'' \emph{Annual Reviews in
  Control}, 2019.

\bibitem{alaviani2018distributed}
S.~S. Alaviani and N.~Elia, ``A distributed algorithm for solving linear
  algebraic equations over random networks,'' in \emph{2018 IEEE Conference on
  Decision and Control (CDC)}.\hskip 1em plus 0.5em minus 0.4em\relax IEEE,
  2018, pp. 83--88.

\bibitem{nedic2010constrained}
A.~Nedi\'c, A.~Ozdaglar, and P.~A. Parrilo, ``Constrained consensus and
  optimization in multi-agent networks,'' \emph{IEEE Transactions on Automatic
  Control}, vol.~55, no.~4, pp. 922--938, 2010.

\bibitem{gade2018acc}
S.~Gade and N.~H. Vaidya, ``Private optimization on networks,'' in \emph{2018
  Annual American Control Conference (ACC)}.\hskip 1em plus 0.5em minus
  0.4em\relax IEEE, 2018, pp. 1402--1409.

\bibitem{gade2018cdc}
------, ``Privacy-preserving distributed learning via obfuscated stochastic
  gradients,'' in \emph{2018 IEEE Conference on Decision and Control
  (CDC)}.\hskip 1em plus 0.5em minus 0.4em\relax IEEE, 2018, pp. 184--191.

\bibitem{huang2015differentially}
Z.~Huang, S.~Mitra, and N.~Vaidya, ``Differentially private distributed
  optimization,'' in \emph{Proceedings of the 2015 International Conference on
  Distributed Computing and Networking}.\hskip 1em plus 0.5em minus 0.4em\relax
  ACM, 2015, p.~4.

\bibitem{zhang2019admm}
C.~Zhang, M.~Ahmad, and Y.~Wang, ``Admm based privacy-preserving decentralized
  optimization,'' \emph{IEEE Transactions on Information Forensics and
  Security}, vol.~14, no.~3, pp. 565--580, 2019.

\bibitem{7431982}
S.~Han, U.~Topcu, and G.~J. Pappas, ``Differentially private distributed
  constrained optimization,'' \emph{IEEE Transactions on Automatic Control},
  vol.~PP, no.~99, pp. 1--1, 2016.

\bibitem{hale2017cloud}
M.~T. Hale and M.~Egerstedt, ``Cloud-enabled differentially private multi-agent
  optimization with constraints,'' \emph{IEEE Transactions on Control of
  Network Systems}, 2017.

\bibitem{alexandru2017privacy}
A.~B. Alexandru, K.~Gatsis, and G.~J. Pappas, ``Privacy preserving cloud-based
  quadratic optimization,'' in \emph{Communication, Control, and Computing
  (Allerton), 2017 55th Annual Allerton Conference on}.\hskip 1em plus 0.5em
  minus 0.4em\relax IEEE, 2017, pp. 1168--1175.

\bibitem{gascon2017privacy}
A.~Gasc{\'o}n, P.~Schoppmann, B.~Balle, M.~Raykova, J.~Doerner, S.~Zahur, and
  D.~Evans, ``Privacy-preserving distributed linear regression on
  high-dimensional data,'' \emph{Proceedings on Privacy Enhancing
  Technologies}, vol. 2017, no.~4, pp. 345--364, 2017.

\bibitem{liu2018gossip1}
Y.~Liu, J.~Wu, I.~R. Manchester, and G.~Shi, ``Gossip algorithms that preserve
  privacy for distributed computation part i: The algorithms and convergence
  conditions,'' in \emph{2018 IEEE Conference on Decision and Control (CDC)},
  2018, pp. 4499--4504.

\bibitem{liu2018gossip2}
------, ``Gossip algorithms that preserve privacy for distributed computation
  part ii: Performance against eavesdroppers,'' in \emph{2018 IEEE Conference
  on Decision and Control (CDC)}, 2018, pp. 5346--5351.

\bibitem{gupta2019statistical}
N.~Gupta, J.~Kat, and N.~Chopra, ``Statistical privacy in distributed average
  consensus on bounded real inputs,'' in \emph{2019 American Control Conference
  (ACC)}.\hskip 1em plus 0.5em minus 0.4em\relax IEEE, 2019, pp. 1836--1841.

\bibitem{gupta2018information}
N.~Gupta, J.~Katz, and N.~Chopra, ``Privacy in distributed average consensus,''
  \emph{IFAC-PapersOnLine}, vol.~50, no.~1, pp. 9515--9520, 2017.

\bibitem{hadjicostis2015robust}
C.~N. Hadjicostis, N.~H. Vaidya, and A.~D. Dom{\'\i}nguez-Garc{\'\i}a, ``Robust
  distributed average consensus via exchange of running sums,'' \emph{IEEE
  Transactions on Automatic Control}, vol.~61, no.~6, pp. 1492--1507, 2015.

\bibitem{gade2017private}
S.~Gade and N.~H. Vaidya, ``Private learning on networks: Part ii,''
  \emph{arXiv preprint arXiv:1703.09185}, 2017.

\bibitem{oliva2016distributed}
G.~Oliva, R.~Setola, and C.~N. Hadjicostis, ``Distributed finite-time
  average-consensus with limited computational and storage capability,''
  \emph{IEEE Transactions on Control of Network Systems}, vol.~4, no.~2, pp.
  380--391, 2016.

\bibitem{shi2015extra}
W.~Shi, Q.~Ling, G.~Wu, and W.~Yin, ``Extra: An exact first-order algorithm for
  decentralized consensus optimization,'' \emph{SIAM Journal on Optimization},
  vol.~25, no.~2, pp. 944--966, 2015.

\bibitem{nedic2017achieving}
A.~Nedi\'c, A.~Olshevsky, and W.~Shi, ``Achieving geometric convergence for
  distributed optimization over time-varying graphs,'' \emph{SIAM Journal on
  Optimization}, vol.~27, no.~4, pp. 2597--2633, 2017.

\bibitem{sun2019convergence}
Y.~Sun, A.~Daneshmand, and G.~Scutari, ``Convergence rate of distributed
  optimization algorithms based on gradient tracking,'' \emph{arXiv preprint
  arXiv:1905.02637}, 2019.

\end{thebibliography}

\end{document}